\documentclass[11pt]{article}
\usepackage{amsthm}
\usepackage{amsmath,amsfonts,amssymb,rotating,colordvi,mdframed}

\usepackage{graphicx,mdframed}
\usepackage{color}
\usepackage{tikz}
\usetikzlibrary{calc}
\usetikzlibrary{shapes.geometric}

\newtheorem{theorem}{Theorem}
\newtheorem{proposition}[theorem]{Proposition}
\newtheorem{corollary}[theorem]{Corollary}

\begin{document}

\title{On properly ordered coloring of vertices in a vertex-weighted graph}

\author{Shinya Fujita\footnote{School of Data Science, Yokohama City University, 22-2 Seto, Kanazawa-ku, Yokohama 236-0027, Japan. {\bf Email:} shinya.fujita.ph.d@gmail.com} \and
Sergey Kitaev\footnote{Department of Mathematics and Statistics, University of Strathclyde, Glasgow, G1 1XH, UK. {\bf Email:} sergey.kitaev@strath.ac.uk} \and Shizuka Sato\footnote{International College of Arts and Sciences, Yokohama City University,  22-2 Seto, Kanazawa-ku, Yokohama 236-0027, Japan. {\bf Email:} shizuka.sato.95@gmail.com} \and 
Li-Da Tong\footnote{Department of Applied Mathematics, National Sun Yat-sen University, No. 70, Lienhai Rd., Kaohsiung 80424, Taiwan. {\bf Email:} ldtong@math.nsysu.edu.tw}}

\maketitle              
\begin{abstract}
We introduce the notion of a properly ordered coloring (POC) of a weighted graph, that generalizes the notion of vertex coloring of a graph. Under a POC, if $xy$ is an edge, then the larger weighted vertex receives a larger color; in the case of equal weights of $x$ and $y$, their colors must be different. 
In this paper, we shall initiate the study of this special coloring in graphs. 
For a graph $G$, we  introduce the function $f(G)$ which gives the maximum number of colors required by a POC over all weightings of $G$. We show that $f(G)=\ell(G)$, where $\ell(G)$ is the number of vertices of a longest path in $G$. 

Another function we introduce is $\chi_{POC}(G;t)$ giving the minimum number of colors required over all weightings of $G$ using $t$ distinct weights. We show that the ratio of $\chi_{POC}(G;t)-1$ to $\chi(G)-1$ can be bounded by $t$ for any graph $G$; in fact, the result is shown by determining $\chi_{POC}(G;t)$ when $G$ is a complete multipartite graph. 

We also determine the minimum number of colors to give a POC on a vertex-weighted graph in terms of the number of vertices of a longest directed path in an orientation of the underlying graph. 
This extends the so called Gallai-Hasse-Roy-Vitaver theorem, a classical result concerning the relationship between the chromatic number of a graph $G$ and the number of vertices of a longest directed path in an orientation of $G$. 

\end{abstract}
\section{Introduction}
In this paper, we consider simple graphs with no loops and no multiple edges. 
By {\em vertex coloring} of a graph, one normally means  assigning integers, called {\em colors}, from $\{1,2,\ldots\}$ to graph's vertices so that no two vertices sharing the same edge have the same color. The smallest number of colors needed to color a graph $G$ is known as its {\em chromatic number}, and is often denoted by $\chi(G)$. There is an extensive literature in graph theory dedicated to vertex coloring and its various generalizations (e.g., see \cite{Malaguti,MalToth} and references therein), and graph coloring has many practical applications, e.g.\ in scheduling~\cite{Marx}, register allocation~\cite{Chaitin}, and in several other areas.   

In this paper, we introduce the notion of a {\em properly ordered coloring} of a weighted graph, or {\em POC}, generalizing the notion of vertex coloring of a graph, and study some of its properties. 

Suppose that $(G,w)$ is a vertex-weighted graph with the vertex set $V(G)$  and the edge set $E(G)$, and  a weight function $w :  V(G)\rightarrow W$, where $W$ is a set of positive integers. Further, let $c : V(G)\rightarrow C$ be a vertex coloring, where $C=\{1,2,\ldots,\theta\}$.  A vertex coloring $c$ on $(G,w)$ is a properly ordered coloring (POC) if and only if for any edge $uv\in E(G)$,
\begin{itemize}
\item if $w(u)>w(v)$ then $c(u)>c(v)$;
\item if $w(u)=w(v)$ then $c(u)\neq c(v)$.
\end{itemize}
We also let $$\chi_{POC}(G,w):=\min\{\theta |\  c \mbox{ is a POC on }(G,w)\}.$$ 

Note that any POC is also a proper vertex coloring on the underlying graph $G$, and so $\chi_{POC}(G,w)\geq \chi(G)$.
 In particular, if $|W|=1$ then  $\chi_{POC}(G,w)= \chi(G)$. Also, note  that for any $(G,w)$, $1\leq \chi_{POC}(G,w)\leq |V(G)|$, and $\chi_{POC}(G,w)$ is well-defined as replacing the weakly ordered weights $w_1\leq w_2\leq\cdots\leq w_{|V(G)|}$, respectively, by $1, 2,\ldots,|V|$ gives a properly ordered coloring.  

Throughout this paper, we may assume that $W=\{1, 2, \ldots ,|W|\}$ and $(G,w)$ contains a vertex $v_j$ such that $w(v_j)=j$ for each $j=1,\ldots ,|W|$ by the following observation:   
In view of the definition of POC, note that, for any pair of vertices $u,v\in V(G)$ with $w(v)<w(u)$, we may ignore the difference value $w(u)-w(v)$ unless there exists a vertex $x\in V(G)$ such that $w(u)<w(x)<w(v)$. (To see this, suppose that $u$ and $v$ are such vertices of $(G,w)$ and $w(u)>w(v)+1$. Let $w'$ be a new weight function obtained from $w$ by changing $w(u)$ as $w'(u)=w(v)+1$ (note that $w'(x)=w(x)$ for all $x\in V(G)-\{u\}$). Obviously, a POC on $(G,w)$ can also be a POC on $(G,w')$.) 

For example, $\chi_{POC}($\hspace{-3.5mm}
\begin{minipage}[c]{3.5em}\scalebox{1}{
\begin{tikzpicture}[scale=0.5]

\draw [line width=1](0,-0.5)--(0,0.5);
\draw [line width=1](1,-0.5)--(1,0.5);
\draw [line width=1](0,-0.5)--(1,-0.5);
\draw [line width=1](0,0.5)--(1,0.5);

\draw (0,-0.5) node [scale=0.4, circle, draw,fill=black]{};
\draw (1,-0.5) node [scale=0.4, circle, draw,fill=black]{};
\draw (0,0.5) node [scale=0.4, circle, draw,fill=black]{};
\draw (1,0.5) node [scale=0.4, circle, draw,fill=black]{};

\node [left] at (0,-0.6){${\small 2}$};
\node [right] at (1,-0.6){${\small 3}$};
\node [left] at (0,0.6){${\small 1}$};
\node [right] at (1,0.6){${\small 1}$};

\end{tikzpicture}
}\end{minipage}
$)=3$ because the induced path $P_3$ given by 1, 2, 3 makes these three vertices require three distinct colors, and the only POC of the graph is \hspace{-3.5mm}
\begin{minipage}[c]{3.5em}\scalebox{1}{
\begin{tikzpicture}[scale=0.5]

\draw [line width=1](0,-0.5)--(0,0.5);
\draw [line width=1](1,-0.5)--(1,0.5);
\draw [line width=1](0,-0.5)--(1,-0.5);
\draw [line width=1](0,0.5)--(1,0.5);

\draw (0,-0.5) node [scale=0.4, circle, draw,fill=black]{};
\draw (1,-0.5) node [scale=0.4, circle, draw,fill=black]{};
\draw (0,0.5) node [scale=0.4, circle, draw,fill=black]{};
\draw (1,0.5) node [scale=0.4, circle, draw,fill=black]{};

\node [left] at (0,-0.6){${\small 2}$};
\node [right] at (1,-0.6){${\small 3}$};
\node [left] at (0,0.6){${\small 1}$};
\node [right] at (1,0.6){${\small 2}$};

\end{tikzpicture}
}\end{minipage}. 

In this paper, we shall investigate a relationship between a graph $G$ and properly ordered coloring on its weighted graphs $(G, w)$. In particular, we would like to find some invariants of a graph $G$ that can have a variety of vertex weight functions $w$. Along this line, we will show that the length of a longest path of a graph can be described as a function in terms of properly ordered coloring. To achieve this, for a graph $G$, let
$$f(G):=\max\{\chi_{POC}(G,w)\ |\ w \mbox{ is a weight function on }G\}.$$
Also, let $\ell(G)$ be the number of vertices of a longest path in $G$. For a longest path $P=p_1\ldots p_{\ell(G)}$ of $G$, 
any weight function $w$ such that $w(p_1)<\cdots <w(p_{\ell(G)})$ forces $\chi_{POC}(G,w)\ge \ell(G)$. Therefore, note that any graph $G$ satisfies $f(G)\ge \ell(G)$. Our first result is to show that, in fact the equality holds. 

\begin{theorem}\label{better-bound-f} Any graph $G$ satisfies $f(G)= \ell(G)$. \end{theorem}

This theorem is shown in Section~\ref{upper-f-G-sec} by providing a greedy algorithm based on vertex ordering given by non-decreasing order of their weights to give a POC on $(G,w)$ with at most $\ell(G)$ colors (see Theorem~\ref{alg-F-prime-theorem} in Section~\ref{upper-f-G-sec}). As an immediate corollary of this theorem, we see that a graph has a weighting requiring the maximum possible number of colors if and only if the graph has a Hamiltonian path. This implies that computing $f(G)$ is NP-hard in general.    

Next we consider a POC on $(G, w)$ with a fixed $|W|$. 
For a positive integer $t$, let 
$$\chi_{POC}(G;t):=\min\{p\ |\ \chi_{POC}(G,w)\leq p \ \mbox{for every} \ w \ \mbox{with} \  |W|=t  \}.$$ 

Note that, by definition, $\chi_{POC}(G;t)\le \chi_{POC}(G;t')$ holds for any pair of $t$ and $t'$ with $t\le t'$ because any weight function $w$ with $|W|=t$ can also be regarded as a weight function with $|W|=t'$ by restricting the image of $w$. 
Somewhat surprisingly, we can show that the ratio of $\chi_{POC}(G;t)-1$ to $\chi(G)-1$ can be bounded by $t$ and the bound is best possible. Indeed, as observed in Section 4, there exist infinitely many graphs that attain the upper bound.   

\begin{theorem}\label{mainth}
For a positive integer $t$, any graph $G$ satisfies 
\begin{center}
$\displaystyle\frac{\chi_{POC}(G;t)-1}{\chi(G)-1}\le t$.
\end{center}
\end{theorem}

In other words, $\chi_{POC}(G;t)$ has a sharp upper bound in terms of $\chi(G)$. 
This theorem is shown by determining $\chi_{POC}(G;t)$ when $G$ is a complete multipartite graph (see Proposition~\ref{p2} in Section~\ref{POC-G-t-sec}). 
We remark that we can easily obtain the weaker statement that $\frac{\chi_{POC}(G;t)}{\chi(G)}\le t$. 
To see this, for a $(G, w)$ with $|W|=t$, let $G_i$ be the induced subgraph by the vertices of weight $i$ in $G$ for $i=1,\ldots ,t$ and give a proper vertex coloring on each $G_i$ so that $\min\{c(x)| x\in V(G_{1})\}=1$ and $\min\{c(x)| x\in V(G_{i+1})\}=\max\{c(x)| x\in V(G_{i})\}+1$ for every $1\le i\le t-1$. Among such vertex colorings, we can find a POC using colors $1,\ldots, t\chi(G)$ on $(G, w)$ because $\chi(G_i)\le \chi(G)$ holds for all $i$. So Theorem~\ref{mainth} is the refinement of this observation up to the tight bound. 

In this paper, we also determine $\chi_{POC}(G,w)$ in terms of the number of vertices of a longest directed path for an acyclic orientation on $G$. To state this, we use the following notation: for a digraph $D$, let $\ell'(D)$ be the number of vertices of a longest directed path in $D$. For a directed path $P'$ consisting of arcs $(p_i,p_{i+1})$ for $i=1,\ldots, |V(P')|-1$, we call $p_1$ and $p_{|V(P')|}$ the \textit{tail} and the \textit{head} of $P'$, respectively. For a vertex weighted graph $(G,w)$, an acyclic orientaion is \textit{good} if $w(x)\geq w(y)$ holds for any arc $(x,y)$ in the orientation; we also define $\ell'(G,w):=\min\{\ell'(D)|\ D$ is a good acyclic orientation on $(G,w)\}$.  

\begin{theorem}\label{main3}
Any weighted graph $(G,w)$ satisfies $\chi_{POC}(G,w)=\ell'(G,w)$. 
\end{theorem}

This theorem can be regarded as the weighted version of the Gallai-Hasse-Roy-Vitaver theorem \cite{Gallai,Hasse,Roy,v}, which states that, for a graph $G$, $\chi(G)$ can be bounded by the number of vertices of a longest directed path in $D$, where $D$ is an orientation of $G$ and the upper bound is attained for an acyclic orientation of $G$.

We now briefly mention an application of the notion of POC to a real world problem. 
In the intelligent chemical processing, graph coloring methods can offer a better way to raise processing effectiveness. For example, there are many stages (or procedures) for some special functions or some restrictions on the relations between two stages in the polyethylene processes (or polymer manufacturing processes) \cite{Holz,Spal}. Machine learning \cite{Kishi,Sch} can help to find the relation between stages and products and graph coloring methods can offer a lower cost way for product. We can set up stages as vertices, relation as arcs, and weights of vertices as the ordering (or prioritization) of all vertices. For example, there are five stages $A$ (stage of watching raw material), $B_1, B_2$ (stages of drying material), $C_1, C_2$ (stages of viscosity), $T$ (stage polymerization), and the relations $AB_1, AB_2, B_1B_2, B_1C_1, B_2C_2, C_1C_2, C_2T$ and a weighted function $w$ (the ordering in the processing) with $w(A)<w(B_1)=w(B_2) <w(C_1)=w(C_2)<w(T)$ in a chemical processing. By POC, we have $c(A)=1, c(B_1)=2, c(B_2)=c(C_1)=3, c(C_2)=4, c(T)=5$. The coloring $c$ can offer the minimum number of steps to run this processing (or to control the flow of this processing).

\section{Upper bounds for $f(G)$}\label{upper-f-G-sec}

In this section, we provide the following rather simple algorithm, thereby proving Theorem~\ref{better-bound-f}. In what follows, the neighbourhood $N(v)$ of a vertex $v$ is the set of all vertices adjacent to $v$. 

\begin{mdframed}

\vspace{3mm}
\centerline{\bf Algorithm F}

\medskip

\noindent
{\bf Input:} $(G,w)$, where $V(G)=\{v_1,\ldots,v_n\}$ and $w(v_1)\leq\cdots\leq w(v_n)$;\\

\noindent
{\bf Output:} a POC with at most $\ell:=\ell(G)$ colors.\\[-3mm]

\begin{description}
\item{{\bf Step 1.}} Set $c(v_1)=1$. 

\item{
{\bf Step 2.}} For $j=2,\ldots ,n$, if $N(v_j)\cap (\cup_{i=1}^{j-1}\{v_i\})=\emptyset$, then set $c(v_j)=1$; otherwise, set  $c(v_j)=\max\{c(v_i)| v_iv_j\in E(G)$ and $1\le i<j\}+1$. 
\end{description}
\end{mdframed}

\begin{theorem}\label{alg-F-prime-theorem} For any $(G,w)$, where $G$ is of order $n$, algorithm F indeed yields a POC with at most $\ell:=\ell(G)$ colors.
\end{theorem}

\begin{proof} By the construction of algorithm F, note that, for any vertex $v$ of $G$, there exists a maximal path $P=p_1p_2\ldots p_t$ in $G$ such that $c(p_1)=1$ and $v=p_t$ and $c(p_{i+1})=c(p_i)+1$ for $i=1,2,\ldots , t-1$. So we have $t\le \ell(G)$.  
Thus, from algorithm F we obtain a POC with at most $\ell$ colors.
\end{proof}

Theorem~\ref{better-bound-f} is an immediate corollary of Theorem~\ref{alg-F-prime-theorem}.

\section{The function $\chi_{POC}(G;t)$ in complete multipartite graphs}\label{POC-G-t-sec}

We start with the following result on $\chi_{POC}(G;t)$  for complete bipartite graphs. 

\begin{theorem}\label{compl-bipart-theorem} For $1\leq m\leq n$, if $t\geq 2m+1$, then $\chi_{POC}(K_{m,n};t)=\min\{m+n, 2m+1\}$.\end{theorem}

\begin{proof} 
If $m\le n\le m+1$, then $K_{m,n}$ contains a Hamiltonian path. In this case, by Theorem~\ref{better-bound-f}, we have $\chi_{POC}(K_{m,n};t)=\min\{m+n, 2m+1\}=m+n$. 

Thus we may assume that $m+2\le n$.  
Since $K_{m,n}$ has a longest path of order $2m+1$, say $P=v_1\ldots v_{2m+1}$, for any $(K_{m,n},w)$ such that $w(v_1)<\cdots<w(v_{2m+1})$, we have $\chi_{POC}(K_{m,n},w)\geq 2m+1$. Thus, $\chi_{POC}(K_{m,n};t)\geq 2m+1$. 

To show that $\chi_{POC}(K_{m,n};t)\leq 2m+1$, let $(X,Y)$ be the partite set of $K_{m,n}$ such that $|X|=m$ and $|Y|=n$. For $(K_{m,n},w)$ we can assume that $X$ and $Y$ can be partitioned into parts $X=X_1\cup\cdots\cup X_m$ and $Y=Y_1\cup\cdots\cup Y_{m+1}$ such that
\begin{itemize}
\item $X_i\neq \emptyset$ for all $1\leq i\leq m$, but some of $Y_i$ can be empty;
\item $\max\{w(x)\ |\ x\in X_i\} \leq \min\{w(x')\ |\ x'\in X_{i+1}\}$ for $i=1,\ldots,m-1$;
\item $\max\{w(y)\ |\ y\in Y_i\} \leq \min\{w(y')\ |\ y'\in Y_{i+1}\}$ for $i=1,\ldots,m$; and
\item $\max\{w(y)\ |\ y\in Y_i\} \leq \min\{w(x)\ |\ x\in X_i\} \leq \max\{w(x)\ |\ x\in X_i\} \leq \min\{w(y)\ |\ y\in Y_{i+1}\}$ for $i=1,\ldots,m$.
\end{itemize}

Let $c:V\rightarrow N$ be the vertex coloring such that $c(x)=2i$ for $x\in X_i$ and $c(y)=2i-1$ for $y\in Y_i$. Note that some of colors in $C$ may not be used because $Y_i$ can be the empty set for some $i$. In any case, by the construction, $c$ is a POC on $(K_{m,n},w)$ such that $|C|\leq 2m+1$. Hence, $\chi_{POC}(K_{m,n};t)\leq 2m+1$ and the theorem is proved.
\end{proof}

We now turn our attention to a more general case, namely, complete multipartite graphs. Unlike Theorem~\ref{compl-bipart-theorem}, it seems difficult to provide a simple formula for $\chi_{POC}(G;t)$ in such general cases. To state our results, we give some preliminaries. 

Suppose that $G$ is a vertex colored graph by $c: V(G)\rightarrow C$ and let $S$ be a subset of $V(G)$ of $s$ vertices. 
If vertices $x_1,\ldots, x_s$ of $S$ can be ordered as $c(x_i)=c(x_{i-1})+1$ for $i=2,\ldots,s$, then we say that $S$ \textit{is consecutively colored}; in particular, when we want to specify the minimum value on $c$, we say that $S$ \textit{is consecutively colored from} $c(x_1)$. Also, for a consecutively colored subset $S$, we sometimes want to specify $x_1$ or $x_s$. In that case, we say that $S$ \textit{is consecutively colored from $x_1$ with color $c(x_1)$ to $x_s$}. Note that we do not need to mention  $c(x_s)$ because we can see the value as long as we know both $c(x_1)$ and $|V(S)|(=s)$ by the assumption that $S$ is consecutively colored.  

Now, for $k\ge 2$, we consider a POC on weighted complete multipartite graphs $(K_{n_1,\ldots, n_k}, w)$. To state our result precisely, we give the following notation. Let  $(K_{n_1,\ldots, n_k}, w)$ be a weighted complete multipartite graph such that $w: V(K_{n_1,\ldots, n_k}) \rightarrow \{1,\ldots ,t\}$.  For $i=1,\ldots, t$, let $H_i$ be a complete subgraph of $K_{n_1,\ldots ,n_k}$ such that $w(x)=i$ for all $x\in H_i$. The subgraphs $H_1,\ldots , H_t$ are called \textit{maximum ordered cliques} (briefly, {\em MOCs}), if $H_1,\ldots , H_t$ are chosen so that $\displaystyle\sum_{i=1}^t|V(H_i)|$ is as large as possible in $K_{n_1,\ldots ,n_k}$. Note that, any MOCs $H_1,\ldots , H_t$ satisfy $H_i\neq\emptyset$ for every $1\le i\le t$ by the assumption that $(G,w)$ contains a vertex $v_j$ such that $w(v_j)=j$ for every $1\le j\le t$. For MOCs $H_1,\ldots , H_t$, we can find a subgraph $\mathcal{S}(H_1,\ldots , H_t)$ of $\overline{K_{n_1,\ldots ,n_k}}$ (the complement of $K_{n_1, \ldots ,n_k}$) having the following three properties (i)--(iii): 

\begin{description}
\item{(i)} Each component $P$ of $\mathcal{S}(H_1,\ldots , H_t)$ forms a path $P=p_1\ldots p_s$ with $s\ge 2$ such that $V(P)$ is contained in a partite set of $K_{n_1,\ldots ,n_k}$ (equivalently, $V(P)$ forms isolated vertices in $K_{n_1,\ldots ,n_k}$) and $w(p_i)=w(p_{i-1})+1$ holds for $i=2,\ldots ,s$ (equivalently, $p_i\in H_{w(p_{i-1})+1}$ holds for $i=2,\ldots ,s$). Moreover, $s$ can be greater than $2$ only if $V(H_{w(p_i)})=\{p_i\}$ holds for every $i$ with $2\le i\le s-1$. (Since $H_1,\ldots ,H_t$ are cliques and $V(P)$ is contained in a partite set of $K_{n_1,\ldots ,n_k}$, note that $|V(H_i)\cap V(P)|\leq 1$ holds for every $1\leq i\leq t$.)
\item{(ii)} For any pair of components $P, P'$ in $\mathcal{S}(H_1,\ldots , H_t)$, there exists at most one clique $H_i$ among $H_1,\ldots ,H_t$ such that $V(P)\cap V(H_i)\neq\emptyset$ and $V(P')\cap V(H_i)\neq\emptyset$.  
\item{(iii)} For $i=1,\ldots ,t$, $|H_i\cap V(\mathcal{S}(H_1,\ldots , H_t))|\le 2$.
\end{description}
  
It is possible that $\mathcal{S}(H_1,\ldots , H_t)=\emptyset$ when there is no such subgraph in $\overline{K_{n_1,\ldots ,n_k}}$ for MOCs $H_1,\ldots , H_t$. Thus we can take $\mathcal{S}(H_1,\ldots , H_t)$ for any MOCs $H_1,\ldots , H_t$. 
By the construction, $\mathcal{S}(H_1,\ldots , H_t)$ forms a union of vertex-disjoint paths. 

As observed in the proofs of Propositions~\ref{p1} and~\ref{p2}, this special subgraph is useful to save the number of colors needed to give a POC on the weighted complete multipartite graph. In fact we will give a vertex coloring so that each component of $\mathcal{S}(H_1,\ldots , H_t)$ has the same color and this is the key idea for saving the number of colors to obtain a desired POC.  

Note that $0\le |V(\mathcal{S}(H_1,\ldots , H_t))|\le 2t-2$, and the upper bound can be attained only if $|V(H_1)\cap V(\mathcal{S}(H_1,\ldots , H_t))|=|V(H_t)\cap V(\mathcal{S}(H_1,\ldots , H_t))|=1$ and $|V(H_i)\cap V(\mathcal{S}(H_1,\ldots , H_t))|=2$ for $i=2,\ldots ,t-1$. 
We say that $\mathcal{S}(H_1,\ldots , H_t)$ is \textit{maximum $(H_1,\ldots ,H_t)$-paths} if it is chosen in such a way that $|V(\mathcal{S}(H_1,\ldots , H_t))|$ is as large as possible in $\overline{K_{n_1,\ldots ,n_k}}$. For maximum $(H_1,\ldots ,H_t)$-paths $\mathcal{S}(H_1,\ldots , H_t)$, let $q(\mathcal{S}(H_1,\ldots , H_t))$ be the number of components in $\mathcal{S}(H_1,\ldots , H_t)$. 

We first prove the following proposition.

\begin{proposition}~\label{p1}
Let $n_1,\ldots, n_k$ be positive integers with $k\ge 2$ and $(K_{n_1,\ldots ,n_k}, w)$ be a vertex weighted graph of $K_{n_1,\ldots ,n_k}$ such that $w: V(G)\rightarrow \{1,\ldots ,t\}$.
Also, let $H_1,\ldots ,H_t$ be a MOCs of $(K_{n_1,\ldots ,n_k}, w)$ and $\mathcal{S}(H_1,\ldots , H_t)$ be a maximum $(H_1,\ldots ,H_t)$-paths of $\overline{K_{n_1,\ldots ,n_k}}$. 
Then, there exists a POC on the graph $(K_{n_1,\ldots ,n_k}, w)$ that uses $\displaystyle\sum_{i=1}^t|V(H_i)|-|V(\mathcal{S}(H_1,\ldots , H_t))|+q(\mathcal{S}(H_1,\ldots , H_t))$ colors. 
\end{proposition}
\begin{proof}
Let $q:=q(\mathcal{S}(H_1,\ldots , H_t))$ and $P_1=p^1_1\ldots p^1_{|V(P_1)|},\ldots, P_q=p^q_1\ldots p^q_{|V(P_q)|}$ be components of $\mathcal{S}(H_1,\ldots , H_t)$ such that 
$w(p^i_1)=\min \{w(x)|\ x\in V(P_i)\}$ and $w(p^i_{|V(P_i)|})=\max \{w(x)|\ x\in V(P_i)\}$ for $i=1,\ldots ,q$ and  $w(p^i_{|V(P_i)|})\le w(p^{i+1}_1)$ for $i=1,\ldots ,q-1$ (so that $P_1,\ldots , P_q$ appear in this order from $H_1$ to $H_t$). 

We now give a vertex-coloring by the following manner: 
For each component $P_i$ of $\mathcal{S}(H_1,\ldots , H_t)$, we assign the same color to all vertices in $P_i$. Therefore, in the following argument, we will only mention the color of one vertex of $P_i$.  

We will basically give a consecutive coloring on $H_i$ from $i=1$ to $t$ successively so that 
$\max\{c(x)|\ x\in V(H_{i-1})\}\le \min\{c(y)|\ y\in V(H_i)\}$ holds for $i=2,\ldots ,t$, where the equality holds only if there exists a component $P$ of $\mathcal{S}(H_1,\ldots , H_t)$ such  that $V(P)\cap V(H_{i-1})\neq\emptyset$ and $V(P)\cap V(H_i)\neq\emptyset$. 
Thus, we start with giving a consecutive coloring on $H_1$. 
If $V(H_1)\cap V(P_1)=\emptyset$ then give an arbitrary consecutive coloring on $H_1$ and continue to give a consecutive coloring on $H_2, \ldots, H_{j-1}$ successively until we have some $H_j$ such that $p^1_1\in V(H_j)$ for the component  $P_1=p^1_1\ldots p^1_{|V(P_1)|}$ of $\mathcal{S}(H_1,\ldots , H_t)$. Note that, if $\mathcal{S}(H_1,\ldots , H_t)=\emptyset$ then we just give consecutive colorings successively on $H_i$ from $i=1$ to $t$ and then we are done. 
So we now assume that $p^1_1\in V(H_j)$ with $j\ge 1$ for $P_1=p^1_1\ldots p^1_{|V(P_1)|}$. 
 If $V(H_j)=\{p^1_1\}$, then we assign color $\displaystyle\sum_{i=1}^{j-1}|V(H_i)|+1$ on $p^1_1$. Otherwise, give a consecutive coloring on $H_j$ from a vertex of $H_1-P_1$ to $p^1_1$. 
From $i=j+1$ to $t$, we successively give a consecutive coloring on $H_i$ by the following manner: For a component $P_i=p^i_1\ldots p^i_{|V(P_i)|}$ of $\mathcal{S}(H_1,\ldots , H_t)$, once we have assigned a color on $p^i_{j-1}$ for some $j\ge 2$, we continue to assign the same color on $p^i_{j}$ until all the vertices of $P_i$ receive the same color. Keeping this coloring procedure in mind, we only have  to consider the following three cases on coloring of $H_i$. 
\begin{description}
\item{$\bullet$}
If $H_i$ does not contain any vertex of $V(\mathcal{S}(H_1,\ldots , H_t))$ that has already been colored just after the coloring procedure on $H_{i-1}$, then we give a consecutive coloring on $H_i$ from a vertex $v\in V(H_i)$ with color $\max\{c(x)|\ x\in \displaystyle\cup_{j=1}^{i-1}V(H_j)\}+1$. In this case, if possible, choose the vertex $v$ so that $v\notin V(\mathcal{S}(H_1,\ldots , H_t))$. (Thus, we would like to color from a vertex $v$ so that $v\notin V(\mathcal{S}(H_1,\ldots , H_t))$ to a vertex in  $V(\mathcal{S}(H_1,\ldots , H_t))$ if it contains a vertex for the next path.) 
\item{$\bullet$} If $H_i$ contains a vertex $v$ of $V(\mathcal{S}(H_1,\ldots , H_t))$ that has already been colored by the coloring procedure and $V(H_i)\cap V(\mathcal{S}(H_1,\ldots , H_t))=\{v\}$, then we give a consecutive coloring on $H_i$ from $v$. 
\item{$\bullet$} If $H_i$ contains a vertex $v$ of $V(\mathcal{S}(H_1,\ldots , H_t))$ that has already been colored by the coloring procedure and $V(H_i)$ contains another (uncolored) vertex $u$ from $V(\mathcal{S}(H_1,\ldots , H_t))$, then we give a consecutive coloring on $H_i$ from $v$ to~$u$. 

\end{description}

Proceeding in this way, in view of the properties (i)--(iii) on $\mathcal{S}(H_1,\ldots , H_t)$, we can give a coloring on all the vertices of $H_1,\ldots ,H_t$. We then assign colors for the other vertices in $K_{n_1,\ldots ,n_k}$. For a vertex $v\in V(K_{n_1,\ldots ,n_k})-V(H_1\cup\ldots \cup H_t)$, there exists a vertex $u\in V(H_j)$ for some $1\le j\le t$ such that $(V(H_j)-\{u\})\cup\{v\}$ induces a clique in $K_{n_1,\ldots ,n_k}$ since $H_1,\ldots , H_t$ are a MOCs. Assign the same color $c(u)$ on $v$.  

We can easily see from the above construction that the resulting coloring is a desired POC. \end{proof}

For the convenience of the readers,  we will now demonstrate giving a POC on $(K_{1,3,5}, w)$ by the coloring procedure described in the poof of Proposition~\ref{p1} in the following case: Let $X=\{x_1\}, Y=\{y_1, y_2, y_3\}, Z=\{z_1, z_2, z_3, z_4, z_5\}$ be the three partite sets of $K_{1,3,5}$ and suppose that $w(x_1)=w(y_1)=w(z_1)=1, w(z_2)=w(z_5)=2, w(y_2)=w(z_3)=3, w(y_3)=w(z_4)=4$. In this case, we can take a MOCs $H_1, H_2, H_3, H_4$ such that $V(H_1)=\{x_1, y_1, z_1\}, V(H_2)=\{z_2\}, V(H_3)=\{y_2, z_3\}, V(H_4)=\{y_3, z_4\}$. Note that $V(K_{1,3,5})-V(H_1\cup H_2\cup H_3\cup H_4)=\{z_5\}$. Let $P_1=z_1z_2z_3, P_2=y_2y_3$ be paths in $\overline{K_{1,3,5}}$.  Then we can let $\mathcal{S}(H_1, H_2, H_3, H_4)=P_1\cup P_2$. 
According to the coloring procedure in the proof of Proposition~\ref{p1}, we can color the vertices as follows: $c(x_1)=1, c(y_1)=2$; any vertex of $P_1$ receives color $3$; any vertex of $P_2$ receives color $4$; $c(z_4)=5, c(z_5)=2$. Note that $q(\mathcal{S}(H_1, H_2, H_3, H_4))=2$ and now we obtained a POC on $(K_{1,3,5}, w)$ using $\sum_{i=1}^4|V(H_i)|-|V(\mathcal{S}(H_1,H_2,H_3, H_4))|+q(\mathcal{S}(H_1,H_2,H_3, H_4))=8-5+2=5$ colors. 

Let $(K_{n_1,\ldots ,n_k}, w)$ be a weighted complete multipartite graph, where $w$ is a weight function such that $w: V(K_{n_1,\ldots ,n_k}) \rightarrow \{1, \ldots ,t\}$. 
For the fixed integers $n_1,\ldots ,n_k$, we can obtain a fixed value $g(n_1,\ldots ,n_k, t; w)$ as follows: 
\begin{center}
 $g(n_1,\ldots ,n_k, t; w)=\max\{\displaystyle\sum_{i=1}^t|V(H_i)|-|V(\mathcal{S}(H_1,\ldots , H_t))|+q(\mathcal{S}(H_1,\ldots , H_t)) \ | \ H_1,\ldots , H_t$ are a MOCs in $(K_{n_1,\ldots ,n_k}, w) \}$. 
\end{center}
Under this notation, we further define the following function on $K_{n_1,\ldots ,n_k}$: 
\begin{center}
$h(n_1,\ldots ,n_k,t):=\max\{g(n_1,\ldots ,n_k,t ;w)|\ w$ is a weight function on $K_{n_1,\ldots ,n_k}$ such that $w: V(K_{n_1,\ldots ,n_k}) \rightarrow \{1, \ldots ,t\}\}$.  
\end{center}
\begin{proposition}~\label{p2}
Let $n_1,\ldots, n_k$ be positive integers with $k\ge 2$. Then, $$\chi_{POC}(K_{n_1,\ldots ,n_k}; t)=h(n_1,\ldots ,n_k,t).$$

\end{proposition}
\begin{proof}
For a weighted complete multipartite graph $(K_{n_1,\ldots ,n_k}, w)$ such that $w: V(K_{n_1,\ldots ,n_k}) \rightarrow \{1, \ldots ,t\}$, let $H_1,\ldots ,H_t$ be a MOCs. By the construction, we need at least  $\displaystyle\sum_{i=1}^t|V(H_i)|-|V(\mathcal{S}(H_1,\ldots , H_t))|+q(\mathcal{S}(H_1,\ldots , H_t))$ colors to give a POC on the induced subgaph by $\displaystyle\cup_{i=1}^tV(H_i)$ in $(K_{n_1,\ldots ,n_k}, w)$, since otherwise, for some $i$, $H_i$ cannot be consecutively colored, or a pair of vertices $x$ and $y$ such that $V(H_i)=\{x\}, V(H_{i+1})=\{y\}$ cannot have distinct colors and then we cannot give a POC on $(K_{n_1,\ldots ,n_k}, w)$. 
This implies that $\chi_{POC}(K_{n_1,\ldots ,n_k}; t)\ge h(n_1,\ldots ,n_k,t)$. 
This, together with Proposition~\ref{p1}, shows that $\chi_{POC}(K_{n_1,\ldots ,n_k}; t)=h(n_1,\ldots ,n_k,t).$
\end{proof}

Now we obtain the following corollary.

\begin{corollary}\label{cor}
Let $n_1,\ldots, n_k$ be positive integers with $k\ge 2$. Then, $$\chi_{POC}(K_{n_1,\ldots ,n_k}; t)\le  (k-1)t+1.$$ 
\end{corollary}
\begin{proof}
Let $H_1,\ldots ,H_t$ and $\mathcal{S}(H_1,\ldots , H_t))$ be as in Proposition~\ref{p1}. 
It suffices to show that $$\displaystyle\sum_{i=1}^t|V(H_i)|-|V(\mathcal{S}(H_1,\ldots , H_t))|+q(\mathcal{S}(H_1,\ldots , H_t))\le (k-1)t+1.$$ 
We will show a contradiction to the assumption that $$\displaystyle\sum_{i=1}^t|V(H_i)|> |V(\mathcal{S}(H_1,\ldots , H_t))|-q(\mathcal{S}(H_1,\ldots , H_t))+(k-1)t+1.$$  
This implies that at least $|V(\mathcal{S}(H_1,\ldots , H_t))|-q(\mathcal{S}(H_1,\ldots , H_t))+2$ cliques of $H_1,\ldots , H_t$ have $k$ vertices, respectively. This assures us that we can find maximum $(H_1,\ldots ,H_t)$-paths in $\overline{K_{n_1,\ldots ,n_k}}$ having at least $|V(\mathcal{S}(H_1,\ldots , H_t))|-q(\mathcal{S}(H_1,\ldots , H_t))+1$ components. Recall that each component of $\mathcal{S}(H_1,\ldots , H_t)$ consists of at least two vertices. Therefore, it contains $$2(|V(\mathcal{S}(H_1,\ldots , H_t))|-q(\mathcal{S}(H_1,\ldots , H_t))+1)$$ vertices, which is more than $|V(\mathcal{S}(H_1,\ldots , H_t))|$ vertices. This contradicts our assumption on the maximality of $|V(\mathcal{S}(H_1,\ldots , H_t))|$.
\end{proof}

The bound on $\chi_{POC}(K_{n_1,\ldots ,n_k}; t)$ of Corollary~\ref{cor} is best possible. We will show this in the next section (after proving Theorem~\ref{mainth}) .  

\section{ Proof of Theorems~\ref{mainth} and~\ref{main3} together with some remarks}
We are now in a position to prove our second main theorem. 
\begin{proof}
Construct a complete multipartite graph $K_{n_1,\ldots ,n_{\chi(G)}}$ from $G$ by adding edges and apply Corollary~\ref{cor} to $K_{n_1,\ldots ,n_{\chi(G)}}$. Then there exists a POC on $K_{n_1,\ldots ,n_{\chi(G)}}$ using at most $(\chi(G)-1)t+1$ colors. This vertex coloring is also a POC on $G$. Thus we have  that $\chi_{POC}(G,t)\le (\chi(G)-1)t+1$. This completes the proof of Theorem~\ref{mainth}. 
\end{proof}

We remark that the upper bound on $\chi_{POC}(G,t)$ is sharp. To see this, consider the case where $(G, w)$ is a weighted complete multipartite graph with $w: V(G)\rightarrow \{1,\ldots ,t\}$ such that each partite set contains vertices having $t$ distinct weights. 
Then any MOCs $H_1,\ldots , H_t$ of $G$ satisfy $$|V(\mathcal{S}(H_1,\ldots , H_t))|=2t-2.$$ Therefore, in view of the coloring procedure described as in the proof of Proposition~\ref{p1}, we can easily check that such a graph attains the upper bound.

We then provide the following two propositions concerning the relationship between digraphs and POC coloring in vertex weighted graphs. 

\begin{proposition}~\label{p3}
Let $(G, w)$ be a weighted graph. Then there
exists a good acyclic orientation of $(G,w)$; moreover, we have $\chi_{POC}(G,w)\leq \ell'(G,w)$.  
\end{proposition}
\begin{proof}
To obtain a desired orientation, we firstly look at the induced subgraph $H_i$ by the vertex set $\{v\in V(G)| w(v)=i\}$ for $i=1,\ldots , |W|$ in $(G, w)$. For every $i$, if $H_i$ contains an edge, then we give any acyclic orientation on $H_i$. We then give an orientation on other edges $xy\in E(G)$ so that $(x,y)$ is an arc if and only if $w(x)>w(y)$ holds.  

We now claim that the resulting orientation $D$ is a desired orientation. To see this, suppose that the resulting digraph $D$ contains a directed cycle $C$. Let $(x,y)$ be an arc on $C$ and consider the directed path obtained from $C$ by deleting $(x,y)$. Note that $y$ is the tail and $x$ is the head on this directed path and hence $w(x)\le w(y)$. By the construction, we may assume that $V(C)$ contains two vertices $u,v$ such that $w(u)<w(v)$. This implies that $w(x)<w(y)$. However, this contradicts the construction of $D$ because we have $(x,y)\in A(D)$. 
Hence $D$ is a good acyclic orientation. To show the second assertion,  we provide the following algorithm F$'$, which is a slight modification of Algorithm F for the digraph case. 
Note that, $N^+(v)$ means the out-neighbour of a vertex $v$ in $D$. 

\begin{mdframed}

\vspace{3mm}
\centerline{\bf Algorithm F$'$}

\medskip

\noindent
{\bf Input:} $(G,w)$ with a good acyclic orientation $D$, where $V(G)=$ \\ $\{v_1,\ldots,v_n\}$ and $w(v_1)\leq\cdots\leq w(v_n)$;\\

\noindent
{\bf Output:} a POC with at most $\ell'(D)$ colors.\\[-3mm]

\begin{description}
\item{{\bf Step 1.}} Set $c(v_1)=1$. 

\item{
{\bf Step 2.}} For $j=2,\cdots ,n$, if $N^+(v_j)\cap (\cup_{i=1}^{j-1}\{v_i\})=\emptyset$, then set $c(v_j)=1$; otherwise, set  $c(v_j)=\max\{c(v_i)| (v_j, v_i)\in A(D)$ and $1\le i<j\}+1$. 
\end{description}
\end{mdframed}
Arguing similarly as in the proof of Theorem~\ref{alg-F-prime-theorem}, we see that Algorithm F$'$ yields a POC with at most $\ell'(D)$ colors, thereby proving that $\chi_{POC}(G,w)\le \ell'(D)$. Thus we have $\chi_{POC}(G,w)\leq \ell'(G,w)$.  
\end{proof}

\begin{proposition}\label{new}
Let $(G, w)$ be a weighted graph and $c: V(G)\rightarrow C$ be a POC on $(G,w)$ with $C=\{1,2, \ldots, \theta\}$. Then there exists a good acyclic orientation $D$ of $G$ such that $\ell'(D)\leq \theta$ (that is, $\chi_{POC}(G,w)\geq \ell'(G,w)$). 
\end{proposition}
\begin{proof}
We give an orientation on $(G,w)$ from the vertex coloring of $c$ on $V(G)$ by the following manner: for an edge $xy$ of $G$, if $c(x)>c(y)$ then orient $xy$ so that $x$ is a tail and $y$ is a head. Let $D$ be the resulting orientation of $(G,w)$. Since $c$ is a POC on $(G,w)$, by the construction, obviously $D$ is a good acyclic orientaton of $(G,w)$ and we need at least $\ell'(D)$ colors in $C$. Thus we have $\ell'(D)\le \theta$.   
\end{proof}


Combining Propositions~\ref{p3} and~\ref{new}, we obtain Theorem~\ref{main3}.

We finally suggest some open questions for graphs with no large clique from the following viewpoint: Obviously, the value of $\chi_{POC}(G,t)$ must be at least the order of any clique contained in $G$. Utilizing our result concerning complete multipartite graphs, we obtained a nontrivial upper bound on $\chi_{POC}(G,t)$ (in terms of the chromatic number and $t$) for general graphs. The situation would change a lot if we restrict our attention to sparse graphs such as planar graphs or graphs with large girth. 
What is the sharp upper bound on $\chi_{POC}(G,t)$ for these graph classes? This could be a challenging but interesting direction of further research. 

\section*{Acknowledgments} The authors would like to thank the referees for carefully reading our article and for many helpful comments.
The first author's research was supported by JSPS KAKENHI (19K03603).

\end{document}